\documentclass[a4paper]{article}
\usepackage{a4, ams, amsmath, amssymb, graphics, subfigure, fullpage, color}%
\usepackage[pdftex]{graphicx}
\usepackage{fancyhdr}
\usepackage{multirow}
\usepackage{amsthm}
\usepackage[colorlinks]{hyperref}
\hypersetup{colorlinks,citecolor=green,filecolor=black,linkcolor=blue,urlcolor=blue}

\newtheorem{corollary}{Corollary}
\newtheorem{proposition}{Proposition}

\newtheorem{example}{Example}

\usepackage{upgreek} 

\newcommand{\MyBib}{D:/Users/fvrins/Documents/LaTeX/MyBib}

\usepackage{lscape}

\newcommand{\InclTDC}

\numberwithin{equation}{section}
\numberwithin{theorem}{section}
\numberwithin{corollary}{section}
\numberwithin{definition}{section}
\numberwithin{lemma}{section}
\numberwithin{remark}{section}
\numberwithin{example}{section}

\DeclareMathOperator{\E}{\mathbb{E}}
\renewcommand{\Q}{\mathbb{Q}}

\DeclareMathOperator{\e}{\textrm{e}}
\DeclareMathOperator{\ind}{1{\hskip -2.5 pt}\textrm{I}}

\newcommand{\beq}{\begin{equation}}
\newcommand{\eeq}{\end{equation}}
\newcommand{\beqn}{\begin{eqnarray}}
\newcommand{\eeqn}{\end{eqnarray}}
\newcommand{\bfig}{\begin{figure}}
\newcommand{\efig}{\end{figure}}
\newcommand{\btab}{\begin{table}}
\newcommand{\etab}{\end{table}}

\ifdefined \up 
	\renewcommand{\up}{\textrm{up}}
\else 
	\DeclareMathOperator{\up}{\textrm{up}} 
\fi

\title{Characteristic Function of Time-Inhomogeneous L\'evy-Driven Ornstein-Uhlenbeck Processes}

\author{Fr\'ed\'eric Vrins\\ Louvain School of Management (LSM)\\ \&\\ Center for Operations Research and Econometrics (CORE)\\ Universit\'e catholique de Louvain\\Chauss\'ee de Binche 151, Office A.212, B-7000 Mons, Belgium\\ \href{mailto:frederic.vrins@uclouvain.be}{frederic.vrins@uclouvain.be}}

\begin{document}
\maketitle

\begin{abstract}
Distributional properties -including Laplace transforms- of integrals of Markov processes received a lot of attention in the literature. In this paper, we complete existing results in several ways. First, we provide the analytical solution to the most general form of Gaussian processes (with non-stationary increments) solving a stochastic differential equation. We further derive the characteristic function of integrals of L\'evy-processes and L\'evy driven Ornstein-Uhlenbeck processes with time-inhomogeneous coefficients based on the characteristic exponent of the corresponding stochastic integral. This yields a two-dimensional integral which can be solved explicitly in a lot of cases. This applies to integrals of compound Poisson processes, whose characteristic function can then be obtained in a much easier way than using joint conditioning on jump times. Closed form expressions are given for gamma-distributed jump sizes as an example.
\end{abstract}








%
\section{Introduction} 
%
Integrals of Markov processes are popular stochastic processes as a results of their numerous applications. Birth, death, bird-death and catastrophe processes in particular received a lot of attention in the context of queuing and storage problems~\cite{Moran59} as well as in biology~\cite{Mang94} or robotics~\cite{Khaluf12}. From a broad perspective, general properties of integrals of Markov processes have been derived, like the time-evolution of associated moments~\cite{Gill92}. With regards to the Laplace transforms in specific, closed form expressions have been derived in~\cite{Poll02} for continuous-time Markov chains taking value on the set of positive integers. With regards to diffusion processes, very general results have been obtained for affine, quadratic and geometric models. The corresponding results are derived from the transition probabilities and the property that the infinitesimal generator is dependent on the space variable only, i.e the stochastic differential equations (SDE) has time-homogeneous coefficients. Explicit formulas for many standard processes including the integrated Orntein-Uhlenbeck (OU), Square-Root and Jacobi diffusions are available in~\cite{Alba05},\cite{Hurd08}. The specific case of integral of geometric Brownian motion, closely linked to squared Bessel processes, has extensively be studied by Yor~\cite{Yor92}, and a collection of papers on the topic is available~\cite{Yor01}. 

In spite of this extensive literature, some important results are still lacking for interesting stochastic processes with sound applications. This is for instance the case of the generalized OU process, that is OU processes with time-varying coefficients and L\'evy driving process. Similarly, to the best of our knowledge, there is no such results for compound Poisson processes. 

In this paper, we are interested in the characteristic function (or equivalently, Laplace transform or moment-generating function -when it exists) of the path integral of the form
\beq
\Lambda_{s,t}=\int_s^t\lambda_u du
\eeq
where $\lambda=(\lambda_t)_{t\geq 0}$ is a Markov process on a probability space $(\Omega,\mathbb{F},\mathcal{F}_t,\Q)$. We assume the sigma-field $\mathcal{F}_t$ contains all the information about $(\lambda,\Lambda)$ up to time $t$. 

The paper is organized as follows. Section 2 is devoted to depict financial applications that will stress the importance of having the characteristic function available in closed form. We shall point out some processes $\lambda_t$ for which the integral is useful in this context. Associated characteristic functions will be derived in the remaining part of the paper. Most general results for Gaussian processes (not necessary with stationary increments) is addressed in Section 3. In Section 4, the corresponding results are obtained for Ornstein-Uhlenbeck processes with a background driving L\'evy process (BDLP) which is not restricted to be a Brownian motion. Finally, Compound Poisson processes are investigated in Section 5 based on joint conditioning of jump times of the underlying Poisson process. Closed-form expressions are given when the jump sizes are Gamma-distributed. 
%
\section{Financial motivation}
%
The most obvious example of the use of integrals of stochastic processes in finance is the pricing of Asian options, in which the payoff depends on the time average of the underlying stock (or index) price. When the underlying stock follows a geometric Brownian motion, as in the standard Black-Scholes case, this leads to the study of integrals of Squared Bessel processes and explains why this specific case received so much attention~(see e.g.~\cite{Geman93},~\cite{Carr04}). In this context, Laplace transforms are interesting when explicit formula for the distributions are too involved or cannot be found. The corresponding distributions can be obtained by simple Fourier inversion, and all the moments are easily retreived by deriving the moment generating function. In some cases however, the explicit form of the characteristic function is very appealing in itself, as it direcly yields the calibration equation, as we now show.

Let $\Q$ stands for the risk-neutral measure and $\lambda_t$ for the risk-free short rate. Then, the time-$t$ price of a risk-free zero-coupon bond paying 1 unit of currency at time $T>t$ is
\beq
B(t,T)=\E[\e^{-\Lambda_{t,T}}|\mathcal{F}_t]\label{eq:ZCB}
\eeq
Consequently, the parameters of any given short rate process can be calibrated at time $t=0$ from the given yield curve $B(0,T)$ by making sure that the Laplace transform $\phi_{\Lambda_{T}}(x):=\E[\e^{-x\Lambda_T}|\mathcal{F}_0]$ of $\Lambda_{T}:=\Lambda_{0,T}$ satisfies
\beq
B(0,T)=\phi_{\Lambda_{T}}(1)
\eeq
which is related to the characteristic function $\varphi_{\Lambda_{T}}(x)$ via $\varphi_{\Lambda_{T}}(x):=\phi_{\Lambda_{T}}(-ix)$.

For tractability reasons, the most popular short-rate models are, by far, Gaussian or square-root diffusion processes with constant coefficients, possibly shifted by a deterministic function. For simple short-rate models like Vasicek (with Ornstein-Uhlenbeck dynamics, OU) or Cox-Ingersoll-Ross (with square root diffusion dynamics, SRD), the analytical expression of the Laplace transform is available; see e.g. \cite{Brigo06},\cite{Cox85a},\cite{Cox85b},\cite{Dufresne90}. However, the corresponding expressions for the Hull-White model, which is the extension of the Vasicek model with time-varying coefficients, is not available.

Expressions similar to~\ref{eq:ZCB} appear in credit risk modelling, in both reduced form and structural approaches (see e.g.~\cite{Lando04},\cite{Biel11}). In reduced form (intensity) approaches, $\tau$ is modeled as the first jump of a Cox process $N_t$ where the stochastic intensity is given by a non-negative stochastic process $\lambda_t$:
\beq
\tau=\inf \{t:N_t>0\}
\eeq
Conditional on the path of the intensity $\lambda$, $N_t$ is a Poisson process, so that
\beq
\Q(\tau>T)=\Q(N_T=0)=\E\left[\e^{-\int_0^T\lambda_s ds}\right]=\E\left[\e^{-\Lambda_T}\right]=\phi_{\Lambda_{T}}(1)\label{eq:cox}
\eeq
The same form of survival probability is obtained in the so-called structural approach. In this setup, the default is modeled as the first passage time of an asset process below a liability threshold $K$:
\beq
\tau=\inf \{t:V_t< K\}\Leftrightarrow \{t:\e^{-V_t}\geq K'\}
\eeq
The process $V_t$ is typically either a Brownian motion, a geometric Brownian motion or a subordinated process. The computation of the law of the first-passage times can be avoided by working in a specific framework where the survival probability takes the same form as in~(\ref{eq:cox}). Let $\Lambda_t$ be a non-decreasing grounded positive process and model the credit event as the first time where $\e^{-\Lambda_t}$ reaches a random threshold $U$ uniformly distributed in $[0,1]$. The passage time is almost surely unique, so that:
\beq
\tau=\inf \{t:\e^{-\Lambda_t} \leq U\}=  \{t:\e^{-\Lambda_t} = U\}
\eeq
Then, the survival probability $\Q(\tau>T)$ takes again the same form:
\beq
\Q(\tau> T)=\Q(\e^{-\Lambda_T} > U)=\E[\Q(\e^{-\Lambda_T} \geq U)]=\E[\e^{-\Lambda_T}]=\phi_{\Lambda_{T}}(1)
\eeq
This suggests that instead of working with an intermediate intensity process, we could directly feed the default model with an non-decreasing L\'evy process $\Lambda_t$. This setup has been used in~\cite{Hull08} in the context of CDO pricing but in a purely numerical setup. The analytical expressions of the calibration equations are derived in~\cite{Vrins11a}. The implied copulas form the class of so-called \textit{Sibuya copulas}, which properties are studied in~\cite{Vrins13}.

This setup, however is not appropriate when one desires to simulate survival probability curves $S(t,T)$. Indeed, 
\beq
\Q(\tau> T|\tau>t)=\E[\e^{-(\Lambda_T-\Lambda_t)}|\mathcal{F}_t]=\E[\e^{-\Lambda_{t,T}}]
\eeq
If the $\Lambda_t$ process is L\'evy, the increments are independent of the past and if they are stationary, $\Lambda_{T}-\Lambda_{t}\sim\Lambda_{T-t}$. Therefore, with $\delta=T-t$
\beq
S(t,t+\delta):=\Q(\tau> t+\delta|\tau>t)=\E[\e^{-\Lambda_{\delta}}]=\phi_{\Lambda_{\delta}}(1)
\eeq
In other words, there is no memory effect. Producing different curves $S(t,t+\delta)$, depending on the path $\Lambda_u,~u\in[0,t]$, can only be introduced by considering non-L\'evy processes for $\Lambda_t$. One way to introduce such dependency is by considering $\Lambda_t$ as the integral of a underlying process $\lambda_t$, as in the intensity process. 
%
\section{Gaussian processes}  
%
As mentioned in the introduction, the expression of integrals of Gaussian processes are known in some specific cases, among which Vasicek (Ornstein-Uhlenbeck) is the most popular one. We generalise here previous results by working out the explicit form of the strong solutions for $\lambda_t$ and $\Lambda_{s,t}$ given $\lambda_s$ when $\lambda_t$ is the solution to the most general SDE associated to Gaussian processes. 

It is very well-known that Gaussian process solving a stochastic differential equation (SDE) has dynamics of the form\footnote{Not all Gaussian processes solve a SDE; this is for instance the case of fractional Brownian motion.}
\beq
d\lambda_t=(\alpha(t)-\beta(t)\lambda_t)dt+\sigma(t)dW_t~~~,~t\geq s,~~\lambda_s\label{eq:SDEg}
\eeq
where $W$ is a Brownian motion and $\alpha(t),\beta(t)$ and $\sigma(t)$ are deterministic integrable functions and $\lambda_s$ the initial condition. When the solution exists, it is possible to derive explicitly the strong solution for $\lambda_t$ and $\Lambda_t$ and obtain the corresponding characteristic function. We assume in this paper that the coefficients here are such that existence and uniqueness conditions are satisfied; see~e.g. \cite{Kara05} or~\cite{Osk03} for a detailed discussion.
\begin{proposition}[General Gaussian Process]\label{prop:gauss} Consider eq.(\ref{eq:SDEg}) and assume existence and uniqueness conditions are met. Suppose further that $\lambda_s$ is known at time $s\leq t$ and define
\beqn
G(s,t):=\e^{-\int_s^t\beta(u)du}~~~~~~~~~~&~~,~~&I(s,t):=\int_{s}^t \alpha(u)G(u,s)du\\
J(s,t):=\int_{s}^t \sigma(u)G(u,s)dW_u&~~,~~&K(s,t):=\int_s^t \sigma(s)G(s,u)du
\eeqn
Then, for any $t\geq s$, the solution to eq.(\ref{eq:SDEg}) conditional upon $\lambda_s$ is
\beq
\lambda_{t}=m(s,t)+G(s,t)J(s,t)\label{eq:lambdaGauss}
\eeq
which is Normally distributed with mean and variance respectively given by
\beqn
m(s,t)&:=&G(s,t)(\lambda_s+I(s,t))\\
v(s,t)&:=&G^2(s,t)\int_s^t \sigma^2(u)G^2(u,s)du=\int_s^t \sigma^2(u)G^2(u,t)du
\eeqn
The integral $\Lambda_{s,t}$ is a shifted Ito integral
\beq
\Lambda_{s,t}=M(s,t)+\int_s^tK(u,t)dW_u\label{eq:LambdaGauss}
\eeq
which is Normally distributed with mean and variance respectively given by
\beqn
M(s,t)&:=&(t-s)\lambda_s+\int_s^t(t-u)\left(\alpha(u)-\beta(u)G(s,u)(\lambda_s+I(s,u))\right)du\\
V(s,t)&:=&\int_s^tK^2(u,t)du
\eeqn
\end{proposition}
\begin{proof} See Appendix~\ref{proof:prop:gauss}.\end{proof}
Since $\Lambda_{s,t}$ is Gaussian with known mean and variance, the following corollary is obvious.
\begin{corollary}[Laplace transform of $\Lambda_{s,t}$]
The Laplace transform of $\Lambda_{s,t}$, conditional upon $\lambda_s$, is
\beq
\phi_{\Lambda_{s,t}}(x)=\exp\left\{-M(s,t)x+\frac{x^2V(s,t)}{2}\right\}\label{eq:GaussPhi}
\eeq
where $M(s,t),V(s,t)$ are given in Proposition~\ref{prop:gauss}.
\end{corollary}
Let us apply this result to get the expression and distribution of $\Lambda_{s,t}$ in some particular cases.
\begin{example}[Integrated rescaled Brownian motion]\label{ex:RescBM}
Consider the case where $\lambda_t=\sigma W_t$. 
 
Then,
\beq
M(s,t)=(t-s)\lambda_s~,~K(s,t)=\sigma(t-s)~,~V(s,t)=\frac{\sigma^2}{3}(t-s)^3
\eeq
This expression can be found directly using the stochastic version of Fubini's theorem on indicator functions.
\end{example}
\begin{example}[Integrated OU process]\label{ex:OU}
The case where $\beta(t)$ and $\sigma(t)$ are constant is widely used in mathematical finance. It corresponds to $\alpha(t)=0 $, $\beta(t)=\beta$ and $\sigma(t)=\sigma$: 
\beqn
M(s,t)&=&\frac{1-\e^{-\beta(t-s)}}{\beta}\lambda_s\\
K(s,t)&=&\frac{\sigma}{\beta}\left(1-\e^{-\beta(t-s)}\right)\\
V(s,t)&=&\left(\frac{\sigma}{2\beta}\right)^2\left(2(t-s)+\frac{\e^{-\beta(t-s)}}{\beta}(4-\e^{-\beta(t-s)})-\frac{6}{\beta}\right)
\eeqn
\end{example}
The above examples correspond to special cases for which results exist (see e.g~\cite{Brigo06}). However, the above result allows to determine the Laplace transform of more general Gaussian processes, as shown in the examples below.
\begin{example}[Prototypical swap exposure]\label{ex:Bridge}
Consider a Brownian bridge from $0$ to $T$ rescaled by a constant $\sigma$ and shifted by the deterministic fontion $\gamma t(T-t)$. Such a process can be written as
\beq
\lambda_t=(T-t)\left(\gamma t+\sigma \int_{0}^t \frac{1}{T-u}dW_{u}\right)~~~,~0\leq s\leq t\leq T
\eeq
This process starts from 0, its expectation follows the curve $\gamma t(T-t)$ and then goes back to $0$ as $t\to T$. This is a Gaussian process with $\alpha(t)=\gamma (T-t)$, $\beta(t)=\frac{1}{T-t}$ zero-mean and instantaneous volatility $\sigma(t)=\sigma$. Then
\beq
M(s,T)=\lambda_t\frac{T-t}{2}+\frac{\gamma}{6}(T-t)^3~,~K(s,T)=\sigma\frac{T-s}{2}~,~V(s,T)=\frac{\sigma^2}{12}(T-s)^3
\eeq
\end{example}
See Appendix~\ref{proof:ex:Bridge} for details.
\begin{example}[Deterministically subordinated process]
Consider the integral $\Lambda_{s,t}=\int_s^t f(s)W_{\theta(s)}ds$ where $\theta(s)$ is a strictly increasing and continuous function. Then,
\beq
\Lambda_{s,t}=\int_{\theta(s)}^{\theta(t)} \frac{f\circ\theta^{-1}(u)}{\theta'(u)}W_{u}du
\eeq
corresponds to the integral of $\lambda_t:=\sigma(t)W_t$ where $\sigma(t):=\frac{f(\theta^{-1}(t))}{\theta'(t)}$. Its differential takes the form~(\ref{eq:SDEg}) where $\alpha(t):=0$, $\beta(t):=-\frac{\sigma'(t)}{\sigma(t)}$ and instantaneous volatility $\sigma(t)$.

Consequently, provided that the existence conditions are met, $\Lambda_{s,t}$ is Normally distributed and the solution is given by~(\ref{eq:LambdaGauss}) where the integration bounds are $s\leftarrow\theta(s)$, $t\leftarrow\theta(t)$ and the parameters are defined above, and the characteristic function follows.\footnote{Observe that when $\theta(t)$ is not strictly increasing (i.e. mereley non-decreasing) and/or discontinuous, the corresponding characteristic function can be recovered by splitting the integral into pieces where the inveres of $\theta(s)$ exists, where $\theta(s)$ is constant and at time points where $\theta(t)$ jumps.}
\end{example}
%
\section{L\'evy-driven Hull-White process}
%
In some cases, the characteristic function $\varphi_{\Lambda_t}(x)$ of the integrated process $\Lambda_t$ can be obtained by computing the characteristic function of a stochastic integral. This is in particular the case of generalized L\'evy-driven OU processes (that is, where the random increments are controlled by a background driving L\'evy process, BDLP), as we now show.

The log-characteristic function of a L\'evy process $X_t$ takes the form 
\beq
\psi_{X_t}(x):=\ln\varphi_{X_t}(x)=t\ln\varphi_{X}(x)=:t\psi_{X}(x)
\eeq
where $\psi_X(x)$ is the \textit{characteristic exponent} of the infinitely divisible distribution of $X:=X_1$~\cite{Cont04}. We can derive the log-characteristic function of the stochastic integral $Y_t$ of a deterministic function $\sigma(t)$ with respect to a L\'evy process $X-t$ as follows. Let $Y_t:=\int_0^t \sigma(s)dX_s$ where $\sigma(s)$ is integrable and $(X_t)_{t\geq 0}$ is L\'evy processes with triplet $(\mu,\sigma,\nu)$ and $\nu(z)$ is the density of the L\'evy  measure. Then
\beqn
\psi_{X}(x)&=&ix\mu+\frac{1}{2}\sigma x^2+\int_\mathbb{R}(\e^{ixz}-1-ixz\ind_{\{|z|<1\}})\nu(dz)\\
\psi_{Y_t}(x)&=&\ln\E[\e^{ix\int_0^t\sigma(s)dX_s}]
\eeqn
\begin{proposition}[Characteristic exponent of a stochastic integral with respect to a L\'evy process]\label{prop:stochint}
Let $(X_t)_{t\geq 0}$ be a L\'evy process with characteristic exponent $\psi_X(x)$ and define the semimartingale $Y_{s,t}=\int_s^t \sigma(u)dX_u$ for some deterministic integrable function $\sigma$. Then, 
\beq
\psi_{Y_{s,t}}(x)=\int_s^t \psi_X(\sigma(u)x)du
\eeq
where $\psi_{Y_{s,t}}(x):=\ln\varphi_{Y_{s,t}}(x)$.
\end{proposition}
\begin{proof} See Appendix \ref{proof:prop:stochint}\end{proof}
The L\'evy driven Hull-White process is defined as a Ornstein-Uhlenbeck process with time-varying coefficients and BDLP. 
\begin{proposition}[Characteristic exponent of integrated BDLP Hull-White process]\label{prop:Levy}
Let $(\lambda_t)_{t\geq 0}$ be a Hull-White process driven by the L\'evy process $X_t$ solution of the SDE
\beq
d\lambda_t=(\alpha(t)-\beta(t)\lambda_{t}) dt +\sigma(t)dX_t\label{eq:sdej}
\eeq
Then, setting $G(s,t):=\e^{-\int_s^t\beta(u)du}$ as before,
\beq
\psi_{\Lambda_{s,t}}(x)=ixM(s,t)+\int_s^t \psi_X\left(x\int_u^t \sigma(u)G(u,v)dv\right)du
\eeq
\end{proposition}
\begin{proof} Using the integration by parts technique used in the proof of Proposition~\ref{prop:gauss}, the solution $\Lambda_{s,t}$ is proven to be the same as in~(\ref{eq:LambdaGauss}) where $W_t$ is replaced by $X_t$:
\beqn
\Lambda_{s,t}=M(s,t)+\int_s^tK(u,t)dX_u
\eeqn
Therefore, 
\beq
\psi_{\Lambda_{s,t}}(x)=ixM(s,t)+\psi_{\int_s^tK(u,t)dX_u}(x)
\eeq
and the claim follows from Proposition~\ref{prop:stochint} and the expression of $K(s,t)$ in Proposition~\ref{prop:gauss}.
\end{proof}
Observe that SDE~(\ref{eq:sdej}) is not the most general SDE with linear drift $\alpha(t)-\beta(t)x$ since the way jumps are handled corresponds to the special case where a jump of $X$ at $t$ of size $z$ triggers a jump of size $\sigma(t)z$ at $t$ for $\lambda_t$. We could perfectly want to have a jump $\lambda_t-\lambda_{t_{-}}$ which does depend on $(\lambda_t,z)$ in a non-linear way (see e.g.~\cite{Bass04} and~\cite{BarndNiel01} for a more advanced discussion and conditions on the coefficients $\alpha(t),\beta(t),\sigma(t)$ and the functional parameters of $X_t$ for existence and pathwise uniqueness).
\begin{example}[OU driven by a Brownian Motion]
We recover the characteristic function of the standard OU model driven by a Brownian motion from the strong solution
\beq
\Lambda_{s,t}=M(s,t)+\int_s^tK(u,t)dW_u
\eeq
where $K(u,t)=\frac{\sigma}{\beta}(1-\e^{-\beta(t-u)})$. Since here $X=W$, the L\'evy triplet is $(0,\sigma,0)$ and the characteristic exponent reduces to $\psi_{X}(x)=-x^2/2$. So,
\beq
\int_s^t \psi_{X}(K(u,t)x)du=-\int_s^t \frac{\left(xK(u,t)\right)^2}{2}du=-\frac{x^2}{2}V(s,t)
\eeq
Therefore,
\beq
\psi_{\Lambda_{s,t}}(x)=ix M(s,t)-\frac{x^2}{2}V(s,t)
\eeq
in line with $\ln \phi_{\Lambda_{s,t}}(-ix)$ in eq.~(\ref{eq:GaussPhi}).
\end{example}
Let us apply this method to determine the characteristic function of $\Lambda_t$ where $\lambda_t$ is a non-Gaussian OU process driven by a gamma process $\gamma_t$.
\begin{example}[OU driven by a gamma process]\label{ex:GammaOU}
Consider the SDE
\beq
d\lambda_t=-\beta\lambda_t dt +d\gamma_t
\eeq
with non-negative solution $\lambda_{t}=\left(\lambda_s+\int_s^t\e^{-\beta (s-u)}d\gamma_u\right)\e^{-\beta (t-s)}$ and the integral is
\beq
\Lambda_{s,t}=\lambda_{s}\frac{1-\e^{-\beta (t-s)}}{\beta}+\int_{s}^t\frac{1-\e^{-\beta (t-u)}}{\beta}d\gamma_u
\eeq
The integrated process $\Lambda_{t}$ is a deterministc constant plus a stochastic integral of a deterministic function $\sigma(u)$ with respect to a L\'evy process. Therefore,
\beq
\psi_{\Lambda_{t}}(x)=ix\lambda_{s}\frac{1-\e^{-\beta (t-s)}}{\beta}+\int_s^t \psi_{\gamma}\left(\frac{1-\e^{-\beta (t-u)}}{\beta}x\right)du
\eeq
where $\psi_\gamma(x)=\psi_\gamma(x;\kappa,\alpha):=\ln\left(\frac{\kappa}{\kappa-ix}\right)^\alpha$ is the characteristic exponent of the gamma distribution driving the jump sizes. The integral can be written in terms of the dilogarithmic function $Li_2(x)$
\beq
(t-s)\psi_\gamma(x/\beta)+\frac{\alpha}{\beta}\left(Li_2\left(v\right)-Li_2\left(v\e^{-\beta (t-s)}\right)\right)
\eeq
where $v=-\frac{ix}{\kappa\beta -ix}$.
\end{example}
A similar approach has been used in~\cite{Eber13}\footnote{We are grateful to D. Madan for providing us with this reference.} to evaluate a joint characteristic function in the $T$-forward measure of a triple of processes which dynamics are given in the risk-neutral measure. 

The following result is a straight consequence of Proposition~\ref{prop:Levy}.
\begin{corollary}[Characteristic exponent of an integrated L\'evy process] Let $(\lambda_t)_{t\geq 0}$ be a L\'evy process. Then,
\beq
\psi_{\Lambda_{s,t}}(x)=ix\lambda_s(t-s)+\int_s^t \psi_\lambda\left(x(t-u)\right)du
\eeq
\end{corollary}
\begin{example}[Stochastically subordinated Brownian motion]
We are interested in the integral of $\lambda_t:=W_{X_t}$ that is, of the integral of a time-changed Brownian motion. Since for every L\'evy process $(X_t)_{t\geq 0}$ we have $\psi_{X_t}(x)=t\psi_X(x)$,
\beq
\varphi_{\lambda_t}(x)=\E\left[\e^{ix\lambda_t}\right]=\E\left[\left(\e^{ixW}\right)^{X_t}\right]=\E\left[\e^{X_t\psi_{W}(x)}\right]=\E\left[\e^{i X_t\left(-i\psi_{W}(x)\right)}\right]=\varphi_{X}^t\left(-i\psi_{W}(x)\right)
\eeq
or equivalently in terms of the characteristic exponent,
\beq
\psi_{\lambda}(x)=\psi_{X}\left(-i\psi_{W}(x)\right)=\psi_{X}\left(ix^2/2\right)
\eeq
In the case of a Gamma subordinator $X_t=\gamma_t$ with parameters $(\kappa,\alpha)$, the subordinated process is will-known to be a variance-gamma (L\'evy) process. Indeed, $\lambda_t$ has the same distribution as the difference of two independent gamma processes, say $\gamma^{(a)}_t$ and $\gamma^{(b)}_t$, each with same parameters $(\alpha,\sqrt{2\kappa})$. This can be seen by noting from the above result that $\varphi_{\lambda_t}(x)=\varphi_{\gamma}^t\left(ix^2/2\right)$:
\beq
\phi_{\lambda_t}(x)=\left(\frac{2\kappa}{2\kappa+x^2}\right)^{\alpha t}=\left(\frac{\sqrt{2\kappa}}{\sqrt{2\kappa}+ix}\right)^{\alpha t}\left(\frac{\sqrt{2\kappa}}{\sqrt{2\kappa}-ix}\right)^{\alpha t}
\eeq
The log-characteristic function of the integral $\Lambda_{s,t}$ of $\lambda_u$ is thus given by
\beqn
\psi_{\Lambda_t}(x)&=&ix\gamma_s(t-s)+\psi_{\int_s^t \gamma_u du}(x)+\psi_{\int_s^t \gamma_u du}(-x)\\
&=&ix\gamma_s(t-s)+\int_s^t\psi_{\gamma}(x(t-u))du+\int_s^t\psi_{\gamma}(x(u-t))du
\eeqn
and
\beq
\int_s^t\psi_{\gamma}(x(t-u))du=\frac{\kappa}{ix}\frac{\psi_\gamma(x(t-s);\kappa,\alpha)}{\varphi_\gamma(x(t-s);\kappa,1)}-\alpha(t-s)(1-\ln\kappa)
\eeq
\end{example}
%
\section{Compound Poisson processes}
%
In this section, we focus on the integral of jump processes where the jumps arrive according to a Poisson process $N=(N_t)_{t\geq 0}$ and the size of the $i$-th jump is given by the random variable $X_i$ which are mutually independent. We then specialize to the case of compound Poisson process where $X_i\sim X$. We compute the Laplace transform $\phi_{\Lambda_t}(x)$ for which the analytical expression can be found for some laws of $X$. This can be done in several ways. We adopt here on a joint conditioning on the number of jumps of $N$ by time $t$ and the corresponding jump times $\vec{T}=(T_1,\ldots,T_{N_t})$.
%
\subsection{Generalized Compound Process}
%
The Laplace transform $\phi_{\Lambda_t}(x)$ of the integrated process $\Lambda_t=\int_0^t\lambda_s ds$ where $\lambda_t=\sum_{i=0}^{N_t}X_i$ is a generalized compound Poisson process is obtained as an infinite sum of multidimensional integrals.
\begin{proposition} Let $N$ ba a Poisson process with intensity $\theta$ and $X_1,X_2,\ldots$ be a sequence of independent random variables. Then, the Laplace transform of the integrated process
\beq
\Lambda_t=\int_0^t\sum_{i=1}^{N_u}X_i du
\eeq
is given by
\beq
\phi_{\Lambda_t}(x)=\e^{-\theta t}\left(1+\sum_{n=1}^{\infty}\theta^n \int_{t_1=0}^t\phi_{X_1}(x(t-t_1))\ldots\int_{t_n=t_{n-1}}^t\phi_{X_n}(x(t-t_n))d\vec{t}_n\right)\label{eq:PhiGCP}
\eeq
\end{proposition}
\begin{proof}
The law of $N_t$ is that of a $Poisson(\theta t)$ random variable, that is
\beq
p_{N_t}(n)=\frac{(\theta t)^n\e^{-\theta t}}{n!},~~n\in\{0,1,2,\ldots\}
\eeq
Furthermore, the laws of the time arrival of the $i$-th jump, $T_i$, and the time $\Delta_i$ elapsed between jumps $i$ and $i+1$ are known
\beqn
\Delta_{i+1}:=T_{i+1}-T_i&\sim& Exp(\theta)=Gamma(1,\theta)\\
T_{i}&=& Gamma(i,\theta)
\eeqn
where we have set $T_0:=0$. Let us further note $\vec{T}_n=(T_1,\ldots,T_n)$ and $\vec{t}_n=(t_1,\ldots,t_n)$. Then, due to the independence between the times separating two consecutive jumps, we compute the joint density $p_{\vec{T}}(\vec{t}_n)$ of the first $n$ jump times $\vec{T}_n$:
\beqn
p_{\vec{T}_n}(\vec{t}_n)=\ind_{\{\vec{t}_n\}}\prod_{i=1}^{n}\theta\e^{-\theta(t_{i}-t_{i-1})}=\ind_{\{\vec{t}_n\}}\theta^n \e^{-\theta t_n}
\eeqn
where $t_0:=0$ and $\ind_{\{\vec{t}_n\}}$ is 1 if $\vec{t}_n$ a non-decreasing sequence of nonnegative numbers and 0 otherwise.

We can then compute the joint density $p_{\vec{T}_n,N_t}(\vec{t}_n,n)$ of $(\vec{T}_n,N_t)$ at $(\vec{t}_n,n)$ using $\{N_t=n\}\Leftrightarrow\{T_n\leq t,T_{n+1}>t\}$ as follows:
\beqn
p_{{\vec{T}}_n,N_t}(\vec{t}_n,n)&=&\ind_{\{\vec{t}_n\}}\Q(\vec{T}_n\in d\vec{t}_n,T_n\leq t,T_{n+1}>t)\\
&=&\ind_{\{\vec{t}_n\}}\ind_{\{t\geq t_n\}}\Q(\vec{T}_n\in d\vec{t}_n,T_{n+1}>t)\\
&=&\ind_{\{\vec{t}_n\}}\ind_{\{t\geq t_n\}}\int_t^\infty p_{\vec{T}_{n+1}}(\vec{t}_n,s)ds\\
&=&\ind_{\{\vec{t}_n\}}\ind_{\{t\geq t_n\}}\theta^n \e^{-\theta t}
\eeqn
This expression depends on $\vec{t}_n$ only via the indicator function $\ind_{\{\vec{t}_n\}}\ind_{\{t\geq t_n\}}$. We can now compute the Laplace transform of $\Lambda_t$ by conditioning on $(\vec{T}_n,N_{t})$ :
\beqn
\phi_{\Lambda_t}(x)&=&\E[\e^{-x\int_0^t\sum_{i=1}^{N_u}X_i}du]\\
&=&p_{N_t}(0)+\sum_{n=1}^{\infty}\E[\e^{-x\sum_{i=1}^{n}X_i(t-t_i)}|\vec{T}_n=\vec{t}_n,N_{t}=n]p_{\vec{T}_n,N_t}(\vec{t}_n,n)\\
&=&p_{N_t}(0)+\sum_{n=1}^{\infty}\int_{t_1=0}^\infty\ldots\int_{t_n=0}^\infty\prod_{i=1}^{n}\phi_{X_i}(x(t-t_i))p_{{\vec{T}},N_t}(\vec{t}_n,n)d\vec{t}_n\\
&=&\e^{-\theta t}+\sum_{n=1}^{\infty}\theta^n \e^{-\theta t}\int_{t_1=0}^t\ldots\int_{t_n=t_{n-1}}^t\prod_{i=1}^{n}\phi_{X_i}(x(t-t_i))d\vec{t}_n
\eeqn
\end{proof} 
This form is appealing for two reasons. First, only the product of the characteristic functions of the jump size variables appear in the multiple integral as $p_{\vec{T}_n,N_t}$ can be sent outside provided that we restrict the integration domain according to the indicator function $\ind_{\{\vec{t}_n\}}\ind_{\{t\geq t_n\}}$. Second, this multiple integral can be written as
\beq
\phi_{\Lambda_t}(x)=\e^{-\theta t}+\e^{-\theta t}\sum_{n=1}^{\infty}\theta^n \underbrace{\int_{t_1=0}^t\phi_{X_1}(x(t-t_1))\ldots\underbrace{\int_{t_n=t_{n-1}}^t\phi_{X_n}(x(t-t_n))}_{:=I_n}}_{I_1}d\vec{t}_n\label{eq:gcp}
\eeq
Therefore, the multiple integral can be computed recursively: only the last characteristic function $\phi_{X_n}$ is involved in the last integral (noted $I_n$) which only depends on time via $(t_{n-1},t)$. The previous integrals, $I_{n-k}$, involves the product of $\phi_{X_{n-k}}(t-s)$ and $I_{n-k}(s,t)$, which limits the complexity of the computation.

Observe that when $x=0$, $\phi_{X_i}(x(t-t_i))=1$ and the multiple integral collapses $t^n/n!$, so that
\beqn
\sum_{n=1}^{\infty}\frac{(t\theta)^n}{n!}=\e^{t\theta}-1
\eeqn
and $\phi_{\Lambda_t}(0)=\e^{-\theta t}(1+\e^{t\theta}-1)=1$ as it should.
\begin{corollary} The Laplace transform of the integrated process
\beq
\Lambda_{s,t}=\int_s^t\sum_{i=1}^{N_u}X_i du
\eeq
is given by
\beq
\phi_{\Lambda_{s,t}}(x)=\e^{-x(t-s)\lambda_{s}}\phi_{\tilde{\Lambda}_{t-s}}(x)
\eeq
where $\phi_{\tilde{\Lambda}_{t}}(x)$ is as $\phi_{\Lambda_{t}}(x)$ in eq.~(\ref{eq:PhiGCP}) but with  $X_i\leftarrow X_{N_s+i}$.
\end{corollary}
\begin{proof}
By definition,
\beq
\int_s^t\sum_{i=1}^{N_u}X_i du=(t-s)\lambda_s+\int_s^t\sum_{i=N_s+1}^{N_u}X_i du=(t-s)\lambda_s+\int_s^t\sum_{i=1}^{N_u-N_s}X_{N_s+i} du
\eeq
The Poisson process has stationary increments, so the distribution of $N_u-N_s\sim N_{u-s}$ and
\beq
\int_s^t\sum_{i=1}^{N_{u}}X_{i} du\sim (t-s)\lambda_s+\int_0^{t-s}\sum_{i=1}^{N_v}X_{N_s+i} dv
\eeq
leading to
\beq
\E\left[\e^{-x\int_s^t\sum_{i=1}^{N_u}X_i du}\right]=\e^{-x(t-s)\lambda_s}\phi_{\tilde{\Lambda}_{t-s}}(x)
\eeq
\end{proof}
%
\subsection{Compound Poisson process with Gamma jump sizes}
%
The above expression for the Generalized Compound Poisson process is appealing but is hard to solve as no recursive formula can be found for solving the multiple integral. Even in the case where the $X_i\sim Exp(\lambda_i)$, the resulting integrals seems intractable. However, when the $X_i$'s are iid, recursion formula can be found as shown below.

Beforehand, observe however that Compound Poisson processes being L\'evy processes with characteristic exponent
\beq
\psi_{\lambda}(x)=\int_\mathbb{R}(\e^{ixz}-1)\nu(dz)
\eeq
where $\nu(dx)=\theta f(dx)$ and $f$ is the jump size distribution we have, from Corollary 2,
\beq
\psi_{\Lambda_{s,t}}(x)=ix\lambda_s(t-s)+\int_s^t\int_{-\infty}^\infty (\e^{izx(t-u)}-1)\nu(dz)dv\label{eq:LevyCP}
\eeq
Solving this simple time-space integral leads the Laplace transform of $\Lambda_{s,t}$. This approach seems considerably simpler than computing the infinite sum of multiple integrals (of arbitrarily large dimension) in eq.~\ref{eq:gcp}.

However, when $X_i\sim Gamma(\kappa,\alpha)$, a recursion formula can be found and the infinite sum of multiple integrals can be solved, and we shall adopt this alternative methodology for the sake of illustration. The results obviously agree with those easily obtained by computing eq.(\ref{eq:LevyCP}) with $\nu(dz)=\theta\frac{\kappa^\alpha}{\Gamma(\alpha)}z^{\alpha-1}\e^{-\kappa z}dz$. 
\begin{example}[Compound Poisson process with exponentially-distributed jump sizes]
Let $N_t$ be a Poisson process with intensity $\theta$ and $X_i\sim Exp(\kappa)=Gamma(\kappa,1)$ be iid exponential variables with Laplace transform $\phi_X(x)=\phi_\gamma(x;\kappa,1)$. Given that $\phi_{\tilde{\Lambda}_{t-s}}(x)=\phi_{\Lambda_{t-s}}(x)$ in the case of iid $X_i$, it is enough to compute $\phi_{\tilde{\Lambda}_{t}}(x)$. Then, the integral $\Lambda_t$ of the compound Poisson process $\lambda_t:=\sum_{i=1}^{N_t} X_i$ is
\beqn
\phi_{\Lambda_t}(x)&=&\left(\e^{-t}\phi_\gamma(xt;\kappa,-\kappa/x)\right)^\theta
\eeqn
\end{example}
To see this, notice that
\beq
\int_{s=t_{n-1}}^t\phi_{\gamma}(x(t-s))ds=-\frac{\kappa}{x}\left[\ln(\kappa+x(t-s))\right]^{t}_{t_{n-1}}=-\frac{\kappa}{x}\ln\phi_\gamma(x(t-t_{n-1}))
\eeq
Integrating that in the previous integral, we get
\beqn
\int_{s=t_{n-2}}^t\phi_{\gamma}(x(t-s))\int_{u=s}^t\phi_{\gamma}(x(t-u))duds&=&-\frac{\kappa}{x}\int_{s=t_{n-2}}^t\phi_{\gamma}(x(t-s))\ln\phi_\gamma(x(t-s))ds\nonumber\\
&=&+\frac{\kappa}{x}\int_{u=0}^{t-t_{n-2}}\phi_{\gamma}(xu)\ln\phi_\gamma(xu)du\\
&=&\left(-\frac{\kappa}{x}\right)^2\int_{u=0}^{\phi_{\gamma}(x(t-t_{n-2}))}\frac{\ln v}{v}dv\\
&=&\frac{1}{2}\left(-\frac{\kappa\ln \phi_{\gamma}(x(t-t_{n-2}))}{x}\right)^2
\eeqn
Finally, using $\int x^{-1}\ln^n x=\frac{\ln^{n+1}x}{n+1}dx$ and $t_0=0$,
\beqn
\int_{t_1=0}^t\phi_{X_1}(x(t-t_1))\ldots\int_{t_n=t_{n-1}}^t\phi_{X_n}(x(t-t_n))d\vec{t}_n&=&\frac{1}{n!}\left(-\frac{\kappa\ln\phi_\gamma(xt)}{x}\right)^n
\eeqn
%
%
\begin{example}[Compound Poisson process with gamma-distributed jump sizes.]
Let $X_i\sim Gamma(\kappa,\alpha)$ be iid gamma random variables with Laplace transform $\phi_X(x)=\phi_\gamma(x;\kappa,\alpha)$ for $x<\kappa$ and $\alpha>1$ (the case $\alpha=1$ corresponds to the exponential above). Then, 
\beqn
\phi_{\Lambda_t}(x)&=&\e^{-\theta\left(t-\frac{\kappa}{x(\alpha-1)}\phi_{\gamma}(xt;\kappa,\alpha-1)\right)}
\eeqn
\end{example}
The above expression can be found easily by noting that
\beq
\int_{s=u}^t\phi_{\gamma}(x(t-s))ds=\int_{s=u}^t\phi_{\gamma}^\alpha(x(t-s);\kappa,1)ds=\frac{\kappa}{x(\alpha-1)}\phi_{\gamma}^{\alpha-1}(x(t-u);\kappa,1)
\eeq
so that
\beqn
\int_{t_1=0}^t\phi_{X_1}(x(t-t_1))\ldots\int_{t_n=t_{n-1}}^t\phi_{X_n}(x(t-t_n))d\vec{t}_n&=&\frac{\left(\kappa/x\right)^{n}\phi_{\gamma}^{n(\alpha-1)}(xt;\kappa,1)}{\prod_{i=1}^n i(\alpha-1)}\nonumber\\
&=&\frac{\left(\kappa\phi_{\gamma}(xt;\kappa,\alpha-1)\right)^{n}}{n!(x(\alpha-1))^n}
\eeqn
%
%
%
\appendix
%
\section{Proof of Proposition~\ref{prop:gauss}}
\label{proof:prop:gauss}
%
It is easy to find the explicit expression of the strong solution $\lambda_{t}$ of eq.~(\ref{eq:SDEg}) with initial value $\lambda_s$, $s\leq t$. Setting,
\beq
G(t,T):=\e^{-\int_t^T\beta(s)ds}=1/G(T,t)
\eeq
we get
\beqn
\lambda_{t}&=&G(s,t)(\lambda_s+I(s,t)+J(s,t))\\
I(s,t)&:=&\int_{s}^t \alpha(u)G(u,s)du\\
J(s,t)&:=&\int_{s}^t \sigma(u)G(u,s)dW_u
\eeqn
In particular, $\lambda_{t}$ given $\lambda_s$ is normally distributed with mean
\beq
m(s,t)=G(s,t)\left(\lambda_s+I(s,t)\right)
\eeq
The variance is easily obtained by noting that $\E[J(s,t)]=0$ and using Ito isometry:
\beqn
v(s,t)&=&G^2(s,t)\int_{s}^t \sigma^2(u)\e^{\int_{s}^u 2\beta(v)dv} du=G^2(s,t)\int_{s}^t \left(\frac{\sigma(u)}{G(s,u)}\right)^2 du=\int_{s}^t \left(\sigma(u)G(u,t)\right)^2 du\nonumber
\eeqn
One can then plug this expression in the integration by parts formula. The resulting expression is again a Ito integral with a deterministic offset:
\beq
\Lambda_{t,T}=T\lambda_T-t\lambda_t-\int_{t}^Tud\lambda_u=(T-t)\lambda_t+\int_{t}^T(T-u)d\lambda_u
\eeq
Using eq.~(\ref{eq:SDEg}), the integral becomes
\beq
\int_{t}^T(T-u)d\lambda_u=\int_t^T(T-u)\left(\alpha(u)-\beta(u)\lambda_u\right)du+\int_t^T(T-u)\sigma(u)dW_u
\eeq
The first integral is
\beqn
\int_t^T(T-u)\left(\alpha(u)-\beta(u)G(t,u)(\lambda_t+I(t,u))\right)du-\int_t^T(T-u)\beta(u)G(t,u)J(t,u)du
\eeqn
All these integrals are deterministic integrals except the last one
\beqn
&&\int_t^T(T-u)\beta(u)G(t,u)\int_t^u\sigma(v)G(v,t)dW_vdu
\\&=&\int_t^T\int_t^u\sigma(v)G(v,t)dW_vd_u\left(\int_t^u(T-s)\beta(s)G(t,s)ds\right)\\
&=&\left(\int_t^T\sigma(v)G(v,t)dW_v\right)\left(\int_t^T(T-s)\beta(s)G(t,s)ds\right)\nonumber\\
&&~~~~~~~~~~~~~~~~-\int_t^T\sigma(u)G(u,t)\left(\int_t^u(T-s)\beta(s)G(t,s)ds\right)dW_u\\
&=&\int_t^T\sigma(u)G(u,t)\left(\int_u^T(T-s)\beta(s)G(t,s)ds\right)dW_u
\eeqn
So the expression for $\Lambda_{t,T}$ is the sum of deterministic functions
\beq
M(t,T)=(T-t)\lambda_t+\int_t^T(T-u)\left(\alpha(u)-\beta(u)G(t,u)(\lambda_t+I(t,u))\right)du
\eeq
and stochastic integrals
\beqn
S(t,T)&=&\int_t^T\sigma(u)\left((T-u)+G(u,t)\int_u^T(s-T)\beta(s)G(t,s)ds\right)dW_u\nonumber\\
&=&\int_t^T\sigma(u)\left(\int_u^T G(u,s)ds\right)dW_u\\
&=&\int_t^TK(u,T)dW_u
\eeqn
setting $K(s,t):=\int_s^t \sigma(s)G(s,u)du$.

The integral $\Lambda_{s,t}$ is Normally distributed with mean $M(s,t)$ given above and variance 
\beqn
V(s,t)&=&\int_s^tK^2(u,t)du
\eeqn
%
\section{Proof for Example~\ref{ex:Bridge}}
\label{proof:ex:Bridge}
%
It is trivial to see that
\beq
G(t,u)=\frac{T-u}{T-t}~,~~~K(t,T)=\sigma\int_{t}^TG(t,u)du=\sigma\frac{T-t}{2}
\eeq
Furthermore,
\beqn
M(t,T)&=&(T-t)\lambda_t+\int_t^T(T-u)\left(\gamma(T-u)-\frac{1}{T-u}G(t,u)(\lambda_t+I(t,u)\right)du\\
&=&(T-t)\lambda_t+\gamma\int_t^T(T-u)^2du-\lambda_t\int_{t}^TG(t,u)du-\int_t^T\int_t^u\gamma(T-v)G(v,u)dvdu\nonumber\\
&=&\lambda_t\frac{T-t}{2}+\frac{\gamma}{3}(T-t)^3-\gamma\int_t^T\int_t^u(T-u)dvdu\\
&=&\lambda_t\frac{T-t}{2}+\frac{\gamma}{3}(T-t)^3-\gamma\int_t^T(T-u)(t-u)du\\
&=&\lambda_t\frac{T-t}{2}+\frac{\gamma}{3}(T-t)^3-\gamma\frac{(T-t)^3}{6}du\\
&=&\lambda_t\frac{T-t}{2}+\frac{\gamma}{6}(T-t)^3
\eeqn
\newpage
%
\section{Proof of Proposition~\ref{prop:stochint}}
\label{proof:prop:stochint}
%
We give the proof for $s=0$ but the development remains valide for any lower bound $0\leq s\leq t$. Every L\'evy process with L\'evy measure $\nu$ can be written as (Prop. 3.7 Cont \& Tankov)
\beq
X_t=\gamma t + \sigma W_t + \int_0^t \int_\mathbb{R} x J_X(ds\times dx)  - \int_0^t \int_{|x|\leq1} x \mu(ds\times dx)
\eeq
where $J_X(ds\times dx)$ is a Poisson random measure with intensity $\mu(ds\times dx)=\nu(dx)ds$. The differential of this L\'evy process is
\beq
dX_s=\gamma ds + \sigma dW_s + \int_\mathbb{R} x J_X(ds\times dx)  - \int_{|x|\leq1} x \mu(ds\times dx)
\eeq
So that for any integrable deterministic function $f(s)$, the stochastic integral $Y_t=\int_0^t f(s)dX_s$ takes the form
\beq
Y_t=\gamma\int_0^t f(s)ds + \sigma \int_0^t f(s) dW_s + \int_0^t \int_\mathbb{R} x f(s) J_X(ds\times dx)  - \int_0^t \int_{|x|\leq1} x f(s) \nu(dx) ds
\eeq
The characteristic exponent of $X=X_1$ is
\beqn
\psi_{X}(u)&=&i\gamma u - \frac{\sigma^2  u^2}{2} + \int_\mathbb{R} (\e^{iux}-1)\nu(dx)  - \int_{|x|\leq1} iux \nu(dx)\\ 
&=&i\gamma u - \frac{\sigma^2 u^2}{2} + \int_\mathbb{R} (\e^{iux}-1- iux\ind_{\{|x|\leq1\}}) \nu(dx)
\eeqn
Therefore, the log-characteristic function $\psi_{Y_t}(u)=\ln \varphi_{Y_t}(u)$ of the stochastic integral $Y_t$ becomes
\beqn
\psi_{Y_t}(u)
&=&\int_0^t \left(i\gamma (uf(s)) -\frac{\sigma^2 (uf(s))^2}{2} + \int_\mathbb{R} \left(\e^{ix (uf(s))}-1 -   ix(uf(s))\ind_{\{|x|\leq1\}}\right) \nu(dx)\right)ds\nonumber\\
&=&\int_0^t \psi_{X}(u f(s))ds
\eeqn
where we have used the exponential formula for Poisson random measure $\E\left[\e^{\int_{[0,T]\times B} g(s,x)J_X(ds\times dx)}\right]=\e^{\int_{[0,T]\times B} (\e^{g(s,x)}-1)\mu(ds\times dx)}$.\footnote{We thank P. Tankov for discussions on this formula when the integral is jointly taken over the time-space set.}
%
\section{Proof for Example~\ref{ex:GammaOU}}
%
Setting $x'=ix$,
\beqn
\psi_{\gamma}\left(\frac{1-\e^{-\beta (t-u)}}{\beta}x\right)&=&\alpha\ln(\kappa\beta)-\alpha\ln\left(\kappa\beta -x'+x'\e^{-\beta (t-u)}\right)\\
&=&\alpha\ln(\kappa\beta)-\alpha\ln\left(1+\frac{x'\e^{-\beta (t-u)}}{\kappa\beta -x'}\right)-\alpha\ln(\kappa\beta -x')\\
&=&\psi_\gamma(x/\beta)-\alpha\ln\left(1+\frac{x'\e^{-\beta (t-u)}}{\kappa\beta -x'}\right)
\eeqn
Set $v:=-\frac{x'\e^{-\beta (t-u)}}{\kappa\beta -x'}$, $v_1:=-\frac{x'\e^{-\beta (t-s)}}{\kappa\beta -x'}$ and $v_2:=-\frac{x'}{\kappa\beta -x'}$
\beqn
\int_s^t\psi_{\gamma}\left(\frac{1-\e^{-\beta (t-u)}}{\beta}x\right)du&=&(t-s)\psi_\gamma(x/\beta)-\alpha\int_s^t\ln\left(1+\frac{x'\e^{-\beta (t-u)}}{\kappa\beta -x'}\right)du\nonumber\\
&=&(t-s)\psi_\gamma(x/\beta)+\frac{\alpha}{\beta}\int_{v_1}^{v_2}\frac{\ln\left(1-v\right)}{v}dv\\
&=&(t-s)\psi_\gamma(x/\beta)+\frac{\alpha}{\beta}\left(Li_2(v_2)-Li_2(v_1)\right)
\eeqn

\ifdefined \MyBib
  \bibliographystyle{plain}
	\bibliography{\MyBib}
\fi
\end{document}